\documentclass{lms}
\title{Borel stability for congruence subgroups}
\author{David Bruce Cohen}

\usepackage{amsmath}
\usepackage{amssymb}
\usepackage{pinlabel}
\usepackage{amsfonts}

\usepackage{amscd}

\DeclareMathOperator{\Stab}{Stab}
\DeclareMathOperator{\Ker}{ker}
\DeclareMathOperator{\SL}{SL}
\DeclareMathOperator{\ssl}{\ensuremath{\mathfrak{sl}}}
\DeclareMathOperator{\GL}{GL}

\DeclareMathOperator{\Sp}{Sp}

\DeclareMathOperator{\ESp}{ESp}

\DeclareMathOperator{\Q}{\ensuremath{\mathbb{Q}}}
\DeclareMathOperator{\Z}{\ensuremath{\mathbb{Z}}}

\DeclareMathOperator{\Mod}{mod}
\DeclareMathOperator{\aaa}{\ensuremath{\mathfrak{a}}}

\DeclareMathOperator{\CC}{C}
\DeclareMathOperator{\RC}{\tilde{C}}
\DeclareMathOperator{\HH}{H}
\DeclareMathOperator{\RH}{\tilde{H}}

\DeclareMathOperator{\Gm}{\Gamma}
\DeclareMathOperator{\sz}{\ensuremath{\sigma_{0}}}
\DeclareMathOperator{\HU}{\ensuremath{\mathcal{HU}}}

\DeclareMathOperator{\THU}{\ensuremath{\tilde{\mathcal{HU}}}}
\DeclareMathOperator{\la}{\langle}
\DeclareMathOperator{\ra}{\rangle}

\DeclareMathOperator{\EE}{\ensuremath{\mathcal{E}}}

\DeclareMathOperator{\hG}{\ensuremath{\hat{G}}}

\DeclareMathOperator{\Id}{Id}

\DeclareMathOperator{\coker}{coker}
\DeclareMathOperator{\im}{im}

\newtheorem{theorem}{Theorem}[section]
\newtheorem{lemma}[theorem]{Lemma}

\newtheorem{assumptionX}{Assumption}

\classno{20H05 (primary),20J06}
\begin{document}

\maketitle

\begin{abstract}
We prove a homological stability theorem for congruence subgroups of symplectic groups.  From this theorem, we deduce a generalization of a theorem of Borel showing that certain homology groups of a congruence subgroup do not depend on the level of the congruence subgroup.
\end{abstract}

\section{Introduction}
\label{section:introduction}
A theorem of Borel \cite[7.5]{Borel} states that for a semisimple real Lie group $G$, the rational group homology $\HH_{m}(\Gm;\Q)$ is the same for any lattice $\Gm\subset G$, at least when $m$ is smaller than some constant depending on $G$.  Consider the case where $G$ is the special linear group $\SL_{n}$.  Given an integer $p\geq 2$, define the level $p$ congruence subgroup $\SL_{n}(\Z,p\Z)$ to be the kernel of the natural map $\SL_{n}(\Z)\rightarrow \SL_{n}(\Z/p\Z)$.  Then for each $m$, Borel \cite[\S 9]{Borel} showed that there exists an integer $N_{m}$ such that $\HH_{m}(\SL_{n}(\Z,p\Z);\Q)\cong \HH_{m}(\SL_{n}(\Z);\Q)$ for all $m>N_{m}$.  This result can be generalized to congruence subgroups defined over rings other than $\Z$.  Please note that throughout the rest of the paper, we only consider two-sided ideals.
\begin{definition*}
For an arbitrary ring $R$ and an ideal $\aaa\subset R$, the congruence subgroup $\GL_{n}(R,\aaa)$ is $\Ker(\GL_{n}(R)\rightarrow \GL_{n}(R/\aaa))$.  
\end{definition*}
The following natural generalization of Borel's theorem is essentially due to Charney\cite{Charney}.
\begin{theorem*}[(Borel stability for congruence subgroups) \cite{Charney}]
Suppose that $R$ is a commutative Noetherian ring of Krull dimension $k-1$ and $\aaa$ is an ideal of $R$ with $R/\aaa$ finite.  Then for $m\geq 1$ and $n\geq 2m+2k+4$, the map $\HH_{m}(\GL_{n}(R,\aaa);\Q)\rightarrow \HH_{m}(\GL_{n}(R);\Q)$ is an isomorphism.
\end{theorem*}

We will prove a version of this theorem for symplectic groups (Theorem \ref{theorem:BorelSymp}).
\begin{definition*}
Suppose that $R$ is a commutative ring and $\aaa$ an ideal of $R$.  The symplectic congruence subgroup $\Sp_{2n}(R,\aaa)$ is $\Ker(\Sp_{2n}(R)\rightarrow\Sp_{2n}(R/\aaa))$.
\end{definition*}

\begin{theorem*}[(Borel stability for congruence subgroups of symplectic groups)]
Suppose that $R$ is a commutative Noetherian ring of Krull dimension $k-2$ and $\aaa$ is an ideal of $R$ with $R/\aaa$ finite.  Then for $m\geq 1$ and $n\geq 2m+2k+4$, the map $\HH_{m}(\Sp_{2n}(R,\aaa);\Q)\rightarrow \HH_{m}(\Sp_{2n}(R);\Q)$ is an isomorphism.
\end{theorem*}

To prove our theorem, we employ the method outlined by Putman \cite[5.11]{Putman}: we first prove a homological stability result for congruence subgroups and then use the stability trick described in \cite[Lemma 5.1]{Putman2}.  A sequence of groups $\Gamma_{1}\subset \Gamma_{2}\subset \Gamma_{3}\subset \ldots$ is said to be homologically stable if the map $\HH_{m}(\Gamma_{n-1})\rightarrow \HH_{m}(\Gamma_{n})$ induced by inclusion $\Gamma_{n-1}\rightarrow \Gamma_{n}$ is an isomorphism for all $n>N_{m}$, where $N_{m}$ is a constant depending on $m$.  This property is known to hold for many natural sequences of groups including symmetric groups \cite{Nakaoka}, general linear groups over rings satisfying stable range conditions \cite{VDK}, and mapping class groups \cite{Harer}.  However, it does not hold for congruence subgroups.  Consider the case where $\Gamma_{n}$ is the congruence subgroup $\SL_{n}(\Z,p\Z)$.  Then Lee and Szczarba \cite{LandS} showed that $\HH_{1}(\Gamma_{n})\cong \ssl_{n}(\Z/p\Z)$ for $n\geq 3$, so homological stability does not hold.  However, if we take homology with rational rather than integral coefficients, then homological stability does hold, i.e., for all $m$ there exists $N_{m}$ such that $\HH_{m}(\Gamma_{n-1};\Q)\rightarrow \HH_{m}(\Gamma_{n};\Q)$ is an isomorphism for $n>N_{m}$, by the aforementioned work of Borel \cite{Borel}.  See Putman \cite{Putman3} for a description of how $\HH_{m}(\Gamma_{n};\Z/p\Z)$ changes as $n$ increases.

More generally, for $\GL_{n}(R,\aaa)$, Charney \cite{Charney} proved that with appropriate coefficients, homological stability will hold for a wide class of rings $R$, for instance any commutative Noetherian ring of finite Krull dimension.  Charney's proof, like most proofs of homological stability theorems, follows a technique introduced in the unpublished work of Quillen which involves examining the action of congruence subgroups on various simplicial complexes.  In this paper, we apply a simplified version of Charney's argument to congruence subgroups of symplectic groups.  In particular, we prove the following, as Theorem \ref{theorem:symp}.

\begin{theorem*}
Suppose that $R$ is a commutative Noetherian ring of Krull dimension $k-2$, and $\aaa$ a proper ideal of $R$.  Suppose further that the inclusion of groups $\aaa\subset R$ induces an isomorphism on homology (with coefficients in a field $M$ with trivial $R$-action):
$$\HH_{\ast}(\aaa;M)\xrightarrow{\cong}\HH_{\ast}(R;M).$$
Then for $n\geq 2m+k+4$, the map
$$\HH_{m}(\Sp_{2(n-1)}(R,\aaa);M)\rightarrow \HH_{m}(\Sp_{2n}(R,\aaa);M)$$
is a surjection, and for $n\geq 2m+k+5$, this map is an isomorphism.
\end{theorem*}

From this theorems, and the above-mentioned stability trick, we will derive our Borel stability theorem for congruence subgroups of symplectic groups.  It is worth noting that we use no analytic techniques, in contrast to Borel.  Also, our proof is not a straightforward generalization of Charney's.  In particular, the complexes considered by Charney depend on $\aaa$, and have no known symplectic analogue.  Defining such complexes and proving that they have the requisite connectivity properties would presumably be quite difficult.  Our techniques allow us to instead use complexes which have been already been studied by Mirzaii and van der Kallen \cite{MVDK}.

\paragraph{Outline.}\S \ref{section:equivhom} reviews the basic equivariant homology needed to operate Quillen's machine for proving homological stability theorems.  \S \ref{section:spec} proves a generic theorem which states that a sequence of groups enjoys homological stability if it acts on a sequence of complexes in an appropriate way.  As usual, the proof proceeds by comparing certain spectral sequences which converge to the equivariant homology of these complexes.  \S \ref{section:symp} shows that the generic theorem can be applied to congruence subgroups of symplectic groups over rings satisfying symplectic stable range conditions acting on complexes of hyperbolic pairs.  Finally \S \ref{section:BorelStab} uses our results to obtain our Borel stability results.

\paragraph{Acknowledgments}The author would like to thank Andrew Putman for his guidance and support, and Ruth Charney, Tom Church, Brendan Hassett, and an anonymous referee for their corrections.

\section{Equivariant Homology}
\label{section:equivhom}
In the next section, we will prove a theorem stating that whenever a sequence of groups acts on a sequence of (simplicial) complexes in a prescribed way, the groups enjoy homological stability.  To do this, we must extensively exploit the theory of equivariant homology, so we review this theory now (a lucid reference is \cite[\S VII.7] {Brown}).
\begin{definition*}
Suppose that a group $G$ acts on the left on a simplicial complex $X$.  Then we can form the equivariant homology of $G$ acting on $X$, denoted $\HH^{G}_{\ast}(X)$, by taking the homology of the chain complex $F_{\ast}\otimes_{G}\CC_{\ast}(X)$ where $F_{\ast}$ is a free resolution of $\Z$ over $\Z[G]$ and $\CC_{\ast}(X)$ is the simplicial chain complex of $X$ (which carries a natural left $G$ action).
\end{definition*}
Here, the operation $\otimes_{G}$ is a slightly more subtle version of tensoring over $\Z[G]$.  Given left $G$-modules $A$ and $B$, the $\Z$-module $A\otimes_{\Z}B$ can be given the diagonal left $G$-action $g\cdot (x\otimes y) = gx \otimes gy$, so we can define $A\otimes_{G} B = (A\otimes B)/G$.  Equivalently, we define a right module structure on $A$ via $x\cdot g = g^{-1}\cdot x$, and take $A\otimes_{\Z[G]}B$.  Of course, we can also compute equivariant homology with coefficients in some abelian group $M$: that is, we define $\HH^{G}_{\ast}(X;M)$ to be the homology of $F_{\ast}\otimes_{G}\CC_{\ast}(X,M)$.  Similarly, we define the reduced equivariant homology $\RH^{G}_{\ast}(X;M)$ to be the homology of $F_{\ast}\otimes_{G}\tilde{C}_{\ast}(X,M)$.

For our purposes, this construction derives its importance from the fact that it can be computed via a spectral sequence which encodes information about the homology of $G$ and its subgroups.  Thus, computing $\RH^{G}_{\ast}(X;M)$ will yield crucial implications for group homology.  The main tool used in this computation is the following lemma \cite[VII.7.3]{Brown}.

\begin{lemma}
\label{lemma:Brown}
Suppose $\RH_{q}(X;M)=0$ for $q=0,...,m$.  Then $\RH^{G}_{q}(X;M)=0$ for $q=0,...,m$.  
\end{lemma}

\paragraph{Remark:} Brown states this as $\HH^{G}_{q}(X;M)=\HH_{q}(G)$, so one might have expected $\RH^{G}_{q}(X;M)$ to be $\RH_{q}(G;M)$.  However, the lemma is true as stated (the sceptical reader should verify this when $X$ is a point).  The point is that $\RH^{G}_{\ast}(X)$ is {\it not} the reduced homology of $F_{\ast}\otimes_{G} C_{\ast}(X,M)$.  We first encountered the use of $F_{\ast}\otimes_{G} \tilde{C}_{\ast}(X,M)$ in Hatcher and Wahl's paper \cite{Hatcher}[p. 26] (although they shrewdly avoid choosing any notation for $\RH^{G}$ itself).  Classical homological stability papers instead use the homology of an acyclic complex which agrees with $F_{\ast}\otimes \tilde{C}_{\ast}(X)$ for small $q$.  One could equally use the non-reduced equivariant homology, at the cost of making certain spectral sequence arguments slightly more subtle.

\paragraph{The spectral sequence.} There is a spectral sequence \cite{Brown}[VII.7.7] with
$$\EE^{1}_{pq}=\bigoplus_{\sigma^{p}\in \Sigma^{p}} \HH_{q}(\Stab_{G}\sigma^{p})$$
where $\Sigma^{p}$ is a collection of representatives of $G$-orbits of $p$-cells of $X$, and by convention there is a single $-1$-cell since $\tilde{C}_{-1}(X;M)=M$, so that $\EE_{-1q}^{1}=\HH_{q}(\Gm_{n})$.  This spectral sequence converges to
$$\RH^{G}_{\ast}(X,M).$$
If $G$ acts ``without rotations", meaning that the stabilizer of any cell of $X$ fixes the cell pointwise, then the differential $d^{1}:\EE^{1}_{pq}\rightarrow \EE^{1}_{(p-1)q}$ can be readily described \cite{Brown}[VII.8].  Suppose that $\hat{\tau}^{p-1}$ is a face of $\sigma^{p}\in \Sigma^{p}$, and $\tau^{p-1}\in \Sigma^{p-1}$ is the representative of its $G$-orbit, so we can write $\tau^{p-1}=g\hat{\tau}^{p-1}$ with $g\in G$.  Then because $G$ acts without rotations: 
$$\Stab_{G}\sigma^{p}\subset \Stab_{G}\hat{\tau}^{p-1} = g^{-1}(\Stab_{G}\tau^{p-1})g$$
We conclude that inclusion followed by conjugation induces a natural map on homology:
$$d^{1}_{\sigma\tau}:\HH_{\ast}(\Stab_{G}\sigma^{p};M)\rightarrow \HH_{\ast}(\Stab_{G}\tau^{p-1};M)$$
The differential $d^{1}:\EE^{1}_{pq}\rightarrow \EE^{1}_{(p-1)q}$ is then just given by summing all of the $d^{1}_{\sigma\tau}$ with appropriate signs.

\section{Spectral sequences}
\label{section:spec}
We will now apply Quillen's machine to prove a ``generic" homological stability theorem.  Fix an abelian group $M$.  From now on, all homology will be taken with coefficients in $M$ unless otherwise noted, i.e., $\HH_{\ast}(X)$ means $\HH_{\ast}(X;M)$.  Now suppose that we have a sequence of groups $G_{1}\subset G_{2}\subset \ldots$, and for each $i$ we have a subgroup $\Gm_{i}\lhd G_{i}$ such that $\Gm_{1}\subset \Gm_{2}\subset \ldots$.  Suppose further that there are simplicial complexes $X_{1},X_{2},\ldots$ such that $G_{i}$ acts on $X_{i}$.  Let $d$ be a positive integer.

\begin{theorem}
\label{theorem:spec}
Assume:\\
1. $G_{n}$ acts transitively on the $p$-cells of $X_{n}$ for $n>p+d$;\\
2. for all $p=0,\ldots,n$, there is a standard $p$-cell $\sigma_{0}^{p}$ in $X_{n}$ such that $\sigma_{0}^{p}$ is always a face of $\sigma_{0}^{p+1}$ and for $n>p+d$ we have $\Stab_{\Gm_{n}}\sigma_{0}^{p}=\Gm_{n-p-1}$ and $\Stab_{\Gm_{n}}\sigma_{0}^{p}$ fixes the vertices of $\sigma_{0}^{p}$.  Furthermore, $\Stab_{G_{n}}\sigma_{0}^{0}=G_{n-1}$ for $n>d+1$;\\
3. $G_{n}$ acts trivially on the image of the map $\HH_{p}(\Gm_{n-1})\rightarrow \HH_{p}(\Gm_{n})$ induced by inclusion $\Gm_{n-1}\rightarrow \Gm_{n}$ for $n>d$ and $p\geq 0$;\\
4. $X_{n}$ and $X_{n}/\Gm_{n}$ are $\frac{n-d}{2}$ acyclic;\\
5. there is some element of $G_{n}$ which flips $\sigma_{0}^{1}$ and commutes with $\Gm_{n-2}$.\\
Then the map $\HH_{m}(\Gm_{n-1})\rightarrow \HH_{m}(\Gm_{n})$ is a surjection for $n\geq 2m+d+1$ and an isomorphism for $n\geq 2m+d+2$.
\end{theorem}

Before proving this theorem, we note that a version of the third assumption was used to a similar end by Charney in \cite{Charney}[Thm. 5.2].

\begin{proof}
We proceed by induction.  When $m=0$, we have that $\HH_{m}(\Gm_{n-1})\rightarrow \HH_{m}(\Gm_{n})$ is just $M\xrightarrow{id} M$ which is clearly an isomorphism.  Now suppose that $m\geq 1$ and assume for induction that for all $q<m$, $\HH_{q}(\Gm_{p-1})\rightarrow \HH_{q}(\Gm_{p})$ is a surjection for $p\geq 2q+d+1$ and an isomorphism for $p\geq 2q+d+2$.  We must show surjectivity and injectivity of $\HH_{m}(\Gm_{n-1})\rightarrow \HH_{m}(\Gm_{n})$ for $n$ in the appropriate ranges.

\paragraph{Surjectivity.} Assume $n\geq 2m+d+1$. Recall our spectral sequence converging to $\RH^{\Gm_{n}}_{\ast}(X_{n})$.  We have
$$\EE^{1}_{pq}=\bigoplus_{\sigma^{p}\in \Sigma^{p}} \HH_{q}(\Stab_{\Gm_{n}}\sigma^{p}),$$
where $\Sigma^{p}$ is a collection of representatives of $\Gm_{n}$-orbits of $p$-cells of $X_{n}$, (again, by convention there is a single $-1$-cell, so $\EE^{1}_{-1q}=\HH_{q}(\Gm_{n})$). In order to establish that $\HH_{m}(\Gm_{n-1})\rightarrow \HH_{m}(\Gm_{n})$ is surjective, it suffices to show that $\EE^{1}_{0m}\rightarrow \EE^{1}_{-1m}$ is surjective since by assumption (3), each summand $\HH_{m}(\Stab_{\Gm_{n}}(\sigma^{0}))$ of $\EE^{1}_{0m}$ has the same image in $\EE^{1}_{-1m}$ (since $G_{n}$ acts transitively on vertices, each $\Stab_{\Gm_{n}}(\sigma^{0})$ is a conjugate of $\Gm_{n-1}$).  To show that $\EE^{1}_{0m}\rightarrow \EE^{1}_{-1m}$ is surjective, it clearly suffices to show that $\EE^{2}_{-1m}=0$.  We will do this by establishing that $\EE^{2}_{pq}$ vanishes for any pair $p,q$ which admits a differential $d^{r}:\EE^{r}_{pq}\rightarrow \EE^{r}_{-1m}$ (this is enough because $\EE^{\infty}_{-1m}=0$ by Lemma \ref{lemma:Brown}).

Observe that the collection of inclusions $\Stab(\Gm_{n})_{\sigma^{p}}\rightarrow \Gm_{n}$ as $\sigma^{p}$ ranges over $\Sigma^{p}$ induces a canonical map $\kappa_{pq}:\EE^{1}_{pq}\rightarrow \RC_{p}(X_{n}/\Gm_{n};\HH_{q}(\Gm_{n}))$ (this map does not depend on our choice of $\Gm_{n}$ orbit representatives $\Sigma^{p}$ because $\Gm_{n}$ acts trivially on its own homology).  If $n-p-1\geq 2q+d+1$, then by our inductive hypothesis, we have a commutative diagram

$$
\begin{CD}
\EE^{1}_{pq}              @<d^{1}_{(p+1)q}<<           \EE^{1}_{(p+1)q}\\
@V\kappa_{pq}VV                               @V\kappa_{(p+1)q}VV\\
\RC_{p}(X_{n}/\Gm_{n};\HH_{q}(\Gm_{n}))     @<d_{p+1}<<
                             \RC_{p+1}(X_{n}/\Gm_{n};\HH_{q}(\Gm_{n}))\\
\end{CD}
$$
\noindent
where the left arrow is an isomorphism and the right arrow is a surjection.

Now suppose that $q<m$ and $p+q=m$ (so that there is a differential from some $\EE^{r}_{pq}$ to $\EE^{r}_{-1m}$). Then
$$n-p-1\geq 2m+d+1-p-1 =m+q+d\geq 2q+d+1,$$
so we have a commutative diagram

$$
\begin{CD}
\EE^{1}_{(p-1)q} @<d^{1}_{pq}<< \EE^{1}_{pq} @<d^{1}_{(p+1)q}<< \EE^{1}_{(p+1)q}\\
@V\cong VV                      @V\cong VV   @V\kappa_{(p+1)q}VV\\
\RC_{p-1}(X_{n}/\Gm_{n};\HH_{q}(\Gm_{n}))@<d_{p}<<
    \RC_{p}(X_{n}/\Gm_{n};\HH_{q}(\Gm_{n}))@<d_{p+1}<<
         \RC_{p+1}(X_{n}/\Gm_{n};\HH_{q}(\Gm_{n}))\\
\end{CD}
$$
where the bottom row is exact by assumption (4) and the rightmost arrow is surjective.  It follows that
$$\EE^{2}_{pq}=\ker(d^{1}_{pq})/\im(d^{1}_{(p+1)q})=0.$$
By our remarks above, this suffices to show surjectivity.

\paragraph{Injectivity.} With our surjectivity result in hand, we now suppose that $n\geq 2m+d+2$. We wish to show that $\HH_{m}(\Gm_{n-1})\rightarrow \HH_{m}(\Gm_{n})$ is injective.  We first confirm that $\EE^{2}_{-1m}=0$ and $\EE^{2}_{0m}=0$.  As before, we note that $\EE^{\infty}_{-1m}=\EE^{\infty}_{0m}=0$ by lemma \ref{lemma:Brown}.  Now, if $p+q\in \{m,m+1\}$, then we see that
$$n-p-1\geq 2m+d+2-p-1\geq m+d+q \geq 2q+d+1$$
so we have that $\EE^{2}_{pq}=0$ as above.

This vanishing result implies that we have an exact sequence:

$$0\leftarrow \EE^{1}_{-1m}{\buildrel d_{0m}^{1}\over \leftarrow} \EE^{1}_{0m}{\buildrel d_{1m}^{1}\over \leftarrow} \EE^{1}_{1m}.$$

Exactness implies that the induced map $d_{0}:\coker(d_{1m}^{1})\rightarrow \EE^{1}_{-1m}=H_{m}(\Gm_{n})$ is an isomorphism.  Hence, recalling yet again that $\Stab_{\Gm_{n}}(\sigma_{0}^{0})=\Gm_{n-1}$, it suffices to show that the natural map $\phi:H_{m}(\Gm_{n-1})\rightarrow \coker(d_{1m}^{1})$ is an isomorphism (since the composition of this map with $d_{0}$ agrees with the map induced by inclusion $\Gm_{n-1}\rightarrow \Gm_{n}$).

We will first show that $\phi$ is a surjection.  Suppose that $\tau_{0}^{0},\tau_{1}^{0}$ are the vertices of some edge $\sigma^{1}$.  Then by assumptions (1) and (5), there is some $g\in G_{n}$ which flips $\sigma^{1}$ and acts trivially on the homology of $\Stab_{\Gm_{n}}(\sigma^{1})$.  The following diagram commutes, and the top arrow is an equality:

$$
\begin{CD}
\HH_{m}(\Stab_{\Gm_{n}}(\sigma^{1}))  @<g <<   \HH_{m}(\Stab_{\Gm_{n}}(\sigma^{1}))\\
@Vd_{\sigma^{1}\tau_{0}^{0}}VV                               @Vd_{\sigma^{1}\tau_{1}^{0}}VV\\
\HH_{m}(\Stab_{\Gm_{n}}(\tau_{0}^{0}))     @<g <<
 \HH_{m}(\Stab_{\Gm_{n}}(\tau_{1}^{0}))\\
\end{CD}
$$

Since $d_{1m}^{1}$ acts as $d_{\sigma^{1}\tau_{0}^{0}}-d_{\sigma^{1}\tau_{1}^{0}}$ on the $\HH_{m}(\Stab_{\Gm_{n}}(\sigma^{1}))$ summand of $\EE^{1}_{0m}$, it follows that when we form $\coker(d_{1m}^{1})$, we identify $\HH_{m}(\Stab_{\Gm_{n}}(\tau^{0}_{0}))$ and $\HH_{m}(\Stab_{\Gm_{n}}(\tau^{0}_{1}))$ by the isomorphism induced by $g$.  Since $X_{n}/\Gm_{n}$ is certainly path connected (by assumption 4,) we quickly see that $\phi$ is surjective, as any two summands of $\EE^{1}_{0m}$ will be identified together by some sequence of edges.

Now we will see that $\phi$ is an injection.  To do this, it suffices to find a map
$$\psi:\coker(d_{1m}^{1})\rightarrow \HH_{m}(\Stab_{\Gm_{n}}(\sigma_{0}^{0}))
=H_{m}(\Gm_{n-1})$$
such that $\psi\circ\phi$ is the identity.  Observe that $\Stab_{G_{n}}(\sigma_{0}^{0})=G_{n-1}$ by assumption 2, and we have already shown that $\HH_{m}(\Gm_{n-2})$ surjects onto $\HH_{m}(\Gm_{n-1})$, so it follows that $\Stab_{G_{n}}(\sigma_{0}^{0})$ acts trivially on $\HH_{m}(\Gm_{n-1})$ by assumption (3). This implies that, for a vertex $\tau^{0}$ and $g_{1},g_{2}\in G_{n}$ such that $g_{1}\tau^{0}=g_{2}\tau^{0}$, we have that $g_{1}$ and $g_{2}$ induce the same map from $\HH_{m}(\Stab_{\Gm_{n}}(\tau^{0}))$ to $\HH_{m}(\Stab_{\Gm_{n}}(g_{i}\tau^{0}))$.  We can now define $\psi$ as a function from $\EE^{1}_{0m}$ as follows: on each summand $\HH_{m}(\Stab_{\Gm_{n}}(\tau^{0}))$, take $\psi$ to be the map to $\HH_{m}(\Stab_{\Gm_{n}}(\sigma_{0}^{0}))$ induced by multiplication by $g\in G$ such that $g\tau^{0}=\sigma_{0}^{0}$.  By what we have just observed, this is well defined and vanishes on the image of $d_{1m}^{1}$.
\end{proof}

\section{Congruence subgroups of symplectic groups}
\label{section:symp}
Now we prove that the homology of $\Sp_{2n}(R,\aaa)$ (with coefficients in an untwisted field) stabilizes if $R$ satisfies a certain symplectic stable range condition.  The stable range condition $SR_{k}$, as defined below, was introduced by Bass.  If a ring $R$ satisfies this condition, then $\GL_{n}(R)$ has various properties.  (See Bass\cite{Bass} or Hahn and O'Meara\cite{HOM} for an exposition).  If we want similar properties to hold for $\Sp_{2n}(R)$, then we will need $R$ to satisfy some additional conditions.

\begin{definition*}
Given a ring $R$, we call a vector $(x_{1},\ldots,x_{k})$ unimodular if there exist $a_{1},\ldots,a_{k}\in R$ such that $a_{1}x_{1}+\ldots+a_{k}x_{k}=1$.  The ring $R$ is said to satisfy the stable range condition $SR_{k}$ if for every unimodular vector $(x_{1},\ldots,x_{k+1})$ we can find $b_{1},\ldots,b_{k}\in R$ such that $(x_{1}+b_{1}x_{k+1},x_{2}+b_{2}x_{k+1},\ldots,x_{k}+b_{k}x_{k+1})$ is unimodular.
\end{definition*}

Any commutative Noetherian ring of Krull dimension $k-1$ satisfies the stable range condition $SR_{k}$ \cite[4.1.10]{HOM}.  Bass also proves the following lemma \cite[V.3.2]{Bass}.

\begin{lemma}
\label{lemma:quotient}
If $R$ satisfies $SR_{k}$, so does $R/\aaa$.
\end{lemma}

Before defining the symplectic stable range condition, we recall the definition of the symplectic group $\Sp_{2n}(R)$, and describe a set of elementary symplectic matrices which play a role analogous to that of elementary matrices in $\GL$.  In the formalism of symplectic algebraic K-theory, it is absolutely standard to assume that $R$ is commutative.
\begin{definition*} Let $R$ be a commutative ring.  Suppose that $x=(x_{1},...,x_{2n})$ and $y=(y_{1},...,y_{2n})$ are vectors in $R^{2n}$.  We say that the symplectic product $\la x,y \ra$ is $x_{1}y_{n+1}+x_{2}y_{n+2}+...+x_{n}y_{2n}- x_{n+1}y_{1}-x_{n+2}y_{2}-...-x_{2n}y_{n}$.
\end{definition*}
The symplectic group $\Sp_{2n}(R)$ is the subgroup of $\GL_{2n}(R)$ which preserves the symplectic form, i.e., $\Sp_{2n}(R)$ is the set of $T\in \GL_{2n}(R)$ such that for all $x,y\in R^{2n}$, we have $\la Tx, Ty \ra = \la x , y \ra$.  Now we will define elementary symplectic matrices.  Let $e_{ij}(r)$ denote the $n\times n$ matrix with all entries $0$ except that entry $i,j$ is equal to $r$.  Let $E_{ij}(r)=\Id_{n}+e_{ij}(r)$.
\begin{definition*}
The following types of matrices are called elementary symplectic matrices (in all cases $i\neq j$):
\begin{align*}
A_{ij}(r)&=
  \begin{bmatrix}
    E_{ij}(r) & 0           \\
    0         & E_{ji}(-r)  \\
  \end{bmatrix} \\
B_{ij}(r)&=
  \begin{bmatrix}
    Id_{n}    & e_{ij}(r)+e_{ji}(r)\\
    0         & Id_{n}             \\
  \end{bmatrix} \\
B_{ii}(r)&=
  \begin{bmatrix}
    Id_{n}    & e_{ii}(r)\\
    0         & Id_{n}             \\
  \end{bmatrix}\\
C_{ij}(r)&=
  \begin{bmatrix}
    Id_{n}                      & 0     \\
    e_{ij}(r)+e_{ji}(r)         & Id_{n}\\
  \end{bmatrix} \\
C_{ii}(r)&=
  \begin{bmatrix}
    Id_{n}                      & 0     \\
    e_{ii}(r)                   & Id_{n}\\
  \end{bmatrix}\\
\end{align*}

The group generated by these types of matrices is called the {\em elementary symplectic group} $\ESp_{2n}(R)$.
\end{definition*}

We can now define the symplectic stable range condition $SpSR_{k}$.  Competing definitions exist, and since we will need a theorem of Mirzaii and van der Kallen\cite{MVDK}, we will use theirs.
\begin{definition*}
A commutative ring $R$ is said to satisfy the symplectic stable range condition $SpSR_{k}$ if it satisfies $SR_{k}$ and $\ESp_{2(k+1)}(R)$ acts transitively on the unimodular vectors of $R^{2(k+1)}$.  This is a special case of the unitary stable range condition of \cite[6.3]{MVDK} (as explained in \cite[Example 6.1]{MVDK}).
\end{definition*}
Importantly, any commutative noetherian ring of Krull dimension $k-1$ satisfies $SpSR_{k}$ \cite[6.5]{MVDK}.
\begin{lemma}
\label{lemma:sympquotient}
If $R$ satisfies $SpSR_{k}$, so does $R/\aaa$.
\end{lemma}
\begin{proof}
The proof of \cite[V.3.2]{Bass} shows that a unimodular vector of  $(R/\aaa)^{2(n+1)}$ lifts to a unimodular vector of $R^{2(n+1)}$.  The lemma thus follows immediately when we note that $\ESp_{2(n+1)}(R)\rightarrow \ESp_{2(n+1)}(R/\aaa)$ is surjective.
\end{proof}

Recall again the definition of a symplectic congruence subgroup.
\begin{definition*}
Suppose that $\aaa$ is an ideal of $R$.  The symplectic congruence subgroup $\Sp_{2n}(R,\aaa)$ is the kernel of the natural map $\Sp_{2n}(R)\rightarrow \Sp_{2n}(R/\aaa)$.
\end{definition*}

Mirzaii and van der Kallen proved a homological stability theorem for $\Sp_{2n}(R)$ when $R$ satisfies a symplectic stable range condition, which paves the way for us to apply theorem \ref{theorem:spec} with $\Gm_{n}=\Sp_{2n}(R,\aaa)$.  In order to do so, we will need a simplicial complex $X_{n}$ acted on by $\Sp_{2n}(R)$ (we will actually take $G_{n}\supset \Gm_{n}$ to be a subgroup of $\Sp_{2n}(R)$).
\begin{definition*}
A partial symplectic basis of $R^{2n}$ is a sequence of pairs of vectors $(a_{1},b_{1}),(a_{2},b_{2}),...,(a_{p},b_{p})$ where $a_{i},b_{i}\in R^{2n}$, and we have that, for all $i,j$ the relations $\la a_{i},a_{j} \ra =0$ and $\la b_{i},b_{j}\ra =0$ hold, and $\la a_{i},b_{j} \ra = \delta^{i}_{j}$.  The {\em poset of partial symplectic bases} $\HU(R^{2n})$ is the poset whose elements are partial symplectic bases.  Order is given by $\sigma<\tau$ whenever $\sigma$ is a subsequence of $\tau$.  The notation $\HU$ arises from the fact that this is a special case of the complex of hyperbolic bases as considered in \cite[\S 7]{MVDK}.

An unordered partial symplectic basis of $R^{2n}$ is a set of pairs of vectors $\{(a_{1},b_{1}),(a_{2},b_{2}),...,(a_{p},b_{p})\}$ given by forgetting the order of a partial symplectic basis.  The {\em complex of unordered partial symplectic bases} $\THU(R^{2n})$ is the simplicial complex whose $p$-simplices correspond to partial symplectic bases of length $p+1$.  Incidence is given by saying that a simplex $\sigma$ is the face of a simplex $\tau$ whenever $\sigma$ is a subsequence of $\tau$.
\end{definition*}

There is an obvious action of $\Sp_{2n}(R)$ on $\THU(R^{2n})$.  Although we will apply our theorem to $\THU$, it is often more convenient to work with $\HU$.  The following lemma will be useful in switching between the two.
\begin{lemma}
\label{lemma:THU}
If $\RH_{q}(\HU(R^{n});M)=0$ for $q=0,\ldots,m$, then
$$\RH_{q}(\THU(R^{n});M)=0$$ for $q=0,\ldots,m$.
\end{lemma}

\begin{proof}
There is a natural map (of posets) $g:\HU(R^{n})\rightarrow \THU(R^{n})$ given by forgetting the order of the elements in a sequence.  As in \cite[Lemma 4.1]{Putman3}, we observe that this map has a right inverse $f:\THU(R^{n})\rightarrow \HU(R^{n})$ given as follows: take any total order $\prec$ on the elements of $R^{n}\times R^{n}$.  Then $f$ maps a simplex $\sigma^{p}=\{(x_{1},y_{1}),\ldots,(x_{p+1},y_{p+1})\}\in \THU$ to
$$((x_{i_{1}},y_{i_{1}}),\ldots,(x_{i_{p+1}},y_{i_{p+1}})),$$
where
$$(x_{i_{1}},y_{i_{1}})\prec(x_{i_{2}},y_{i_{2}})\prec\ldots\prec (x_{i_{p+1}},y_{i_{p+1}}).$$  The map $f$ is easily seen to be order preserving.  Since $g\circ f$ is the identity, $g$ induces a surjection on homology.
\end{proof}

We now state our main theorem.  Again, we stress that the coefficient field $M$ in consideration is untwisted.
\begin{theorem}
\label{theorem:symp}
Suppose that $R$ is a ring satisfying the symplectic stable range condition $SpSR_{k}$, and $\aaa$ a proper ideal of $R$.  Suppose further that the inclusion of groups $\aaa\subset R$ induces an isomorphism $\HH_{\ast}(\aaa;M)\xrightarrow{\cong}\HH_{\ast}(R;M)$ and that $M$ is a field. Then for $n\geq 2m+k+4$, the map
$$\HH_{m}(\Sp_{2(n-1)}(R,\aaa);M)\rightarrow \HH_{m}(\Sp_{2n}(R,\aaa);M)$$
is a surjection, and for $n\geq 2m+k+5$, this map is an isomorphism.
\end{theorem}
\begin{proof}
From here on, $\HH_{\ast}(X)$ will mean $\HH_{\ast}(X;M)$.  If $\aaa=R$, then the theorem follows by applying \cite[Theorem 8.2]{MVDK} in the symplectic case \cite[Example 6.1]{MVDK}.  Hence we can assume $\aaa\neq R$.

We will apply Theorem \ref{theorem:spec} with
\begin{itemize}
\item $G_{n}=\langle \ESp_{2n}(R), \Sp_{2n}(R,\aaa)\rangle\subset\Sp_{2n}(R)$
\item $\Gm_{n}=\Sp_{2n}(R,\aaa)$
\item $X_{n}=\HU(R^{2n})$, $d=k+3$.
\end{itemize}
(We take the map $G_{n}\rightarrow G_{n+1}$ to be lower-right inclusion).  We must verify that all the assumptions are valid.

\begin{assumptionX}
$G_{n}$ acts transitively on $p$-simplices of $X_{n}$ for $p<n-d$.
\end{assumptionX}
\cite[Lemma 7.1]{MVDK} proves that $\ESp_{2n}(R)$ acts transitively on $p$-simplices of $\HU(R^{2n})$ for $p<n-d$, which suffices to prove the assumption.

\begin{assumptionX}
For all $p=0,\ldots,n$, there is a standard $p$-simplex $\sigma_{0}^{p}$ in $X_{n}$ such that, for $n>p+d$, $\Stab_{\Gm_{n}}\sigma_{0}^{p}=\Gm_{n-p-1}$, and furthermore, $\Stab_{\Gm_{n}}\sigma_{0}^{p}$ fixes the vertices of $\sigma_{0}^{p}$.  Additionally, $\Stab_{G_{n}}(\sigma_{0}^{0})=G_{n-1}$ for $n>d+1$
\end{assumptionX}
The standard $p$-cell $\sigma_{0}^{p}$ is the standard hyperbolic partial basis:
$$\sigma_{0}^{p}=\{(e_{1},e_{n+1}),(e_{2},e_{n+2}),\ldots (e_{p+1},e_{n+p+1})\}$$
Since $\aaa\neq R$, a matrix in $\Sp_{2n}(R,\aaa)$ which permutes $e_{1},\ldots, e_{p+1}$ must fix these vectors (as the identity is the only permutation matrix with trivial image in $\Sp_{2n}(R/\aaa)$).  We thus have that the stabilizer of $\sigma_{0}^{p}$ in $\Sp_{2n}(R,\aaa)$ consists of all matrices of the form:
\[ \begin{bmatrix} \Id_{p+1} & 0 & 0 & 0 \\ 0 & \ast & 0 & \ast \\ 0 & 0 & \Id_{p+1} & 0\\ 0 & \ast & 0 & \ast\\ \end{bmatrix}\in  \Sp_{2n}(R,\aaa)\]
which is clearly just (the lower-right-included image of) $\Sp_{2(n-p-1)}(R,\aaa)$, as desired.  

Finally, let $S=\Stab_{G_{n}}(\sigma_{0}^{0})$.  As above, we see that $S=G_{n}\cap\Sp_{2(n-1)}(R)$.  By definition, $S\supset G_{n-1}$.  On the other hand, it is easy to see that
$$\langle S\cup \ESp(2n)(R)\rangle \subset G_{n}=
\langle G_{n-1}\cup \ESp(2n)(R)\rangle$$
Hence, by injective stability of symplectic $K_{1}$ (see \cite[Corollary 3.2]{stein}\cite{HOM}), we get the reverse inclusion $S\subset G_{n-1}$, and thus $S=G_{n-1}$ as desired.

\begin{assumptionX}
The group $G_{n}$ acts trivially on the image of the map $\HH_{p}(\Gm_{n-1})\rightarrow \HH_{p}(\Gm_{n})$ induced by (lower-right) inclusion $\Gm_{n-1}\rightarrow \Gm_{n}$ for $n>d$ and $p\geq 0$.
\end{assumptionX}
It suffices to show that elementary symplectic matrices act trivially on the image $H$ of $\HH_{p}(\Sp_{2(n-1)}(R,\aaa))\rightarrow \HH_{p}(\Sp_{2n}(R,\aaa))$ for $n>d$.  We first prove a small lemma:
\begin{lemma}
The group $\ESp_{2n}(R)$ is generated by the collection of all matrices of the following types.
\begin{itemize}
\item $A_{1j}(r)$ (with $j=2,\ldots n$)
\item $A_{i1}(r)$ (with $i=2,\ldots,n$)
\item $B_{1j}(r)$ (with $j=1,\ldots n$)
\item $C_{i1}(r)$ (with $i=1,\ldots n$)
\end{itemize}
\end{lemma}
\begin{proof}
For alternating permutations $\pi\in A_{n}$, let $T_{\pi}\in \GL_{n}(R)$ be such that $T_{\pi}e_{i}=e_{\pi(i)}$.  (Because $\pi$ is alternating, $(T_{\pi}^{-1})^{T}=T_{\pi}$).  Let $E_{ij}$ denote $E_{ij}(1)$.  It is well known that the elementary matrices $E_{ij}$ generate $\SL_{n}(\Z)$, and the map $\SL_{n}(\Z)\rightarrow \GL_{n}(R)$ preserves both the alternating permutation matrices and the $E_{ij}$, so we conclude that the matrix $T_{\pi}\in \GL_{n}(R)$ is a product of matrices of the form $E_{ij}$.  But all the elementary matrices in $\GL_{n}(R)$ are products of matrices of the forms $E_{i1}(r)$ and $E_{1j}(r)$, (for instance, $E_{13}(r)=[E_{12},E_{23}(r)]$,) so it follows that the matrix:
\[
S_{\pi}=
\begin{bmatrix}
T_{\pi} & 0       \\
0       & T_{\pi} \\
\end{bmatrix}
\]
can be generated by matrices of the forms $A_{1j}(1)$ and $A_{i1}(1)$.  But then by conjugating matrices of types $A_{1j}(r), A_{i1}(r), B_{1j}(r)$ and $C_{1j}(r)$ by the $S_{\pi}$, we can generate every elementary symplectic matrix.  For instance $B_{23}(r)=S_{(123)}B_{12}(r)S_{(123)}^{-1}$.
\end{proof}
Now we will show that matrices of types $A_{1j}(r)$ and $B_{1j}(r)$ act trivially on $H$.  The case for $A_{i1}(r)$ and $C_{i1}(r)$ is similar.  Let:
\[
G=\left\{
\begin{bmatrix}
1 & \ast & \ast & \ast\\
0 & A    & \ast & B   \\
0 & 0    & 1    & 0   \\
0 & C    & \ast & D   \\
\end{bmatrix}
\in \Sp_{2n}(R,\aaa)
\right\}
\]
\[
\hG=\left\{
\begin{bmatrix}
1 & \ast & \ast & \ast\\
0 & A    & \ast & B   \\
0 & 0    & 1    & 0   \\
0 & C    & \ast & D   \\
\end{bmatrix}
\in \Sp_{2n}(R):
\begin{bmatrix} A & B\\ C & D\\ \end{bmatrix} \in \Sp_{2(n-1)}(R,\aaa)
\right\}.
\]\\
(The reader may wish to think of $G$ as the subgroup of $\Sp_{2n}(R,\aaa)$ which preserves the basis vector $e_{1}$.  Let $P$ be the subgroup of $\Sp_{2n}(R)$ which preserves $e_{1}$, and observe that there is a natural map $P\rightarrow \Sp_{2(n-1)}(R)$ given by letting $T\in P$ act on $v$ in the span of $\{e_{i}\}_{i\neq 1,n+1}$ then throwing away the $e_{1}$ and $e_{n+1}$ components. The preimage of $\Sp_{2(n-1)}(R,\aaa)$ under this map is $\hG$). 

We have that $A_{1j}(r),B_{1j}(r)\in \hG$, so these matrices act trivially on $\HH_{p}(\hG)$.  On the other hand $\Sp_{2(n-1)}(R,\aaa)\subset G \subset \Sp_{2n}(R,\aaa)$, so it suffices to show that $A_{1j}(r),B_{1j}(r)$ act trivially on $\HH_{p}(G)$.  This will follow once we show that $\HH_{p}(G)\rightarrow \HH_{p}(\hG)$ is an isomorphism.  Let $K$ be the kernel of the natural map $G\rightarrow \Sp_{2(n-1)}$ described above.  (In coordinates, this map is given by forgetting rows $1,n+1$ and columns $1,n+1$).  Similarly, let $\hat{K}=\Ker(\hG\rightarrow\Sp_{2(n-1)})$.  Then we have a commutative diagram:
$$\begin{CD}
1     @>>>    K  @>>> G   @>>> \Sp_{2(n-1)} @>>> 1    \\
@.         @VVV      @VVV      @|                @.   \\
1     @>>>\hat{K}@>>> \hG @>>> \Sp_{2(n-1)} @>>> 1    \\
\end{CD}$$

In order to show that the middle downward arrow induces an isomorphism on homology, we will need the following lemma, together with the fact (to be proved shortly) that $K\rightarrow \hat{K}$ induces an isomorphism on homology with $M$ coefficients.  Recall that $M$ is a field, and $\HH_{\ast}(\cdot)$ refers to homology with coefficients in $M$ unless otherwise specified.

\begin{lemma}
\label{lemma:diagram}
Suppose that we have a diagram of groups, with exact rows,
$$\begin{CD}
1    @>>> N_{1}   @>>> G_{1}    @>>> Q_{1}   @>>> 1 \\
@.        @VVV         @VVV      @VVV       @.\\
1    @>>> N_{2}   @>>> G_{2}    @>>> Q_{2}   @>>> 1 \\
\end{CD}$$
Suppose further that $N_{1}\rightarrow N_{2}$ induces an isomorphism on homology, and either condition (i) or (ii) below holds.
\begin{itemize}
\item[(i)] $Q_{1}=Q_{2}$;
\item[(ii)] for all $p$, the map $Q_{1}\rightarrow Q_{2}$ induces an isomorphism
$$\HH_{p}(Q_{1})\rightarrow H_{p}(Q_{2})$$
and $Q_{i}$ acts trivially on $\HH_{p}(N_{i})$.
\end{itemize}
Then we have that
$$\HH_{p}(G_{1})\rightarrow \HH_{p}(G_{2})$$
is an isomorphism.
\end{lemma}
\begin{proof}
Given a short exact sequence of groups $1\rightarrow N \rightarrow G \rightarrow Q$, there is a spectral sequence called the Hochschild-Serre spectral sequence \cite[VII.6]{Brown} with $E^{2}_{pq}=\HH_{p}(Q;\HH_{q}(N))$ and $E^{\infty}_{pq}\Rightarrow \HH_{p+q}(G)$.  This spectral sequence is natural, in the sense that our diagram induces a map from the spectral sequence associated to its top row to the spectral sequence associated to its bottom row.  On $E^{2}_{pq}$ this takes the form of the natural map
$$\HH_{p}(Q_{1};\HH_{q}(N_{1}))\rightarrow \HH_{p}(Q_{2};\HH_{q}(N_{2})).$$
If (i) holds, then this map is an isomorphism, and hence the spectral sequences have isomorphic abutments, i.e.,
$$\HH_{p}(G_{1})\cong \HH_{p}(G_{2})$$
as desired.  On the other hand, if (ii) holds, then we reason as follows.  Since $M$ is a field, we can use the universal coefficient theorem to write the map on the $E^{2}$ page as
$$\HH_{p}(Q_{1})\otimes \HH_{q}(N_{1})\rightarrow \HH_{p}(Q_{2})\otimes \HH_{q}(N_{2}).$$
Since this map is an isomorphism, we see that the map of abutments is an isomorphism, i.e., $\HH_{p}(G_{1})\cong \HH_{p}(G_{2})$.
\end{proof}

By the lemma, we will be done if we can show that $K\rightarrow \hat{K}$ induces an isomorphism on homology.  But this follows from applying the lemma to the commutative diagram:
$$\begin{CD}
1     @>>> \aaa   @>>> K     @>>> \aaa^{2(n-1)} @>>> 1    \\
@.         @VVV      @VVV        @|                @.   \\
1     @>>> R     @>>>\hat{K}@>>> R^{2(n-1)}   @>>> 1    \\
\end{CD}$$
(Observe that condition (ii) holds here).

\begin{assumptionX}
The spaces $X_{n}$ and $X_{n}/\Gm_{n}$ are $\frac{n-d}{2}$ connected.
\end{assumptionX}
\cite[Lemma 7.4]{MVDK} shows that $\HU(R^{2n})$ is $\frac{n-k-3}{2}$ connected for $R$ satisfying $SpSR_{k}$.  So it suffices to prove that the $n-d$-skeleton $(\THU(R^{2n})/\Gm_{n})^{(n-d)}=\THU((R/\aaa)^{2n})^{(n-d)}$.  We want to show that $\Sp_{2n}(R,\aaa)$ orbits of unordered partial hyperbolic bases of $R^{2n}$ correspond to $\aaa$-congruence classes of unordered partial hyperbolic bases of $R^{2n}$.  That is, given $\sigma^{p},\tau^{p}$ partial hyperbolic bases of length $p+1$, we want that:
$$\sigma^{p}\equiv \tau^{p} \Mod \aaa \iff \sigma^{p}\in \Sp_{2n}(R,\aaa)\tau^{p}$$

The $\Leftarrow$ implication is trivial, so we now prove the $\Rightarrow$ implication.  It suffices to prove this for ordered, rather than unordered, partial hyperbolic bases.  Let $\sigma_{0}^{p}$ denote the standard hyperbolic partial basis
$$((e_{1},e_{n+1}),(e_{2},e_{n+2}),\ldots (e_{p+1},e_{n+p+1})).$$
We must show that for any length $p+1$ partial hyperbolic basis $\tau^{p}=((x_{1},y_{1}),\ldots,(x_{p+1},y_{p+1}))\equiv \sz^{p} \Mod \aaa$, there exists $T\in \Sp_{2n}(R,\aaa)$ such that $T\sigma_{0}^{p}=\tau^{p}$.  As we observed before, \cite{MVDK} proves that there is some $S\in \ESp_{2n}(R)$ such that $S\sz^{p}=\tau^{p}$ and thus $S$ has the form
\[ \begin{bmatrix} x_{1} & \ldots & x_{p+1} & \ast & y_{1} & \ldots & y_{p+1} & \ast\\ \end{bmatrix} \]
Let $\bar{S}$ be the image of $S$ under $\Sp_{2n}(R)\rightarrow \Sp_{2n}(R/\aaa)$. Then $\bar{S}$ has the form:
\[
\begin{bmatrix}
\Id_{p+1} & 0 & 0          & 0\\
0         & A & 0          & B\\
0         & 0 & \Id_{p+1}  & 0\\
0         & C & 0          & D\\
\end{bmatrix}
\]
By \cite[9.1.11]{HOM}, the map:
$$\Sp_{2(n-p-1)}(R/\aaa)/\ESp_{2(n-p-1)}(R/\aaa)\rightarrow \Sp_{2n}(R/\aaa)/\ESp_{2n}(R/\aaa)$$
induced by lower-right inclusion is an injection since $n-p-1$ exceeds the stable range of $R/\aaa$.  So, since $\bar{S}\in \ESp_{2n}(R/\aaa)$, we conclude that $\begin{bmatrix} A & B\\ C & D \\ \end{bmatrix}\in \ESp_{2(n-p-1)}(R/\aaa)$.  But $\ESp_{2n}(R)\rightarrow \ESp_{2n}(R/\aaa)$ is surjective, so we can lift $\begin{bmatrix} A & B\\ C & D \\ \end{bmatrix}$ to some matrix $\tilde{S}\in \ESp_{2(n-p-1)}(R)\subset\ESp_{2n}(R)$  Let $T=S\tilde{S}^{-1}$.  Then clearly $T\in \Sp_{2n}(R,\aaa)$, and $T\sz^{p}=S\tilde{S}^{-1}\sz^{p}=S\sz^{p}=\tau^{p}$ as desired.
\begin{assumptionX}
There is some $g\in G_{n}$ which flips $\sigma^{1}_{0}$ and commutes with $\Gm_{n-2}$.
\end{assumptionX}
If $n=2$, then this is trivial, so assume $n\geq 3$. Let $\pi$ be the transposition $(1\quad 2)$ and (writing $A_{ij}$ for $A_{ij}(1)$ and $B_{ij}$ for $B_{ij}(1)$) take $g\in \ESp_{2n}(R)$ to be the matrix
\[
\begin{bmatrix}
T_{\pi} & 0 & 0          & 0\\
0         & \Id_{n-2} & 0 & 0\\
0         & 0 & T_{\pi}  & 0\\
0         & 0 & 0        & \Id_{n-2}\\
\end{bmatrix}
=A_{12}^{-1}A_{21}A_{12}^{-1}(B_{22}C_{22}^{-1}B_{22})^{2}
\]
Since $g$ commutes with $\Gm_{n-2}$ and flips $\sigma_{0}^{1}$, this completes the verification of the final assumption, and thus, the proof.
\end{proof}

\section{Borel Stability}
\label{section:BorelStab}
Using theorem \ref{theorem:symp} (and its proof) we will prove that, with appropriate qualifiers, the homology of a symplectic congruence subgroup does not depend on its level.

\begin{theorem}
\label{theorem:BorelSymp}
Suppose that $R$ is a commutative ring satisfying the symplectic stable range condition $SpSR_{k}$, and $\aaa$ an ideal of $R$, with $R/\aaa$ finite.  Then for $m\geq 1$ and $n\geq 2m+2k+4$, the map:
$$\HH_{m}(\Sp_{2n}(R,\aaa);\Q)\rightarrow \HH_{m}(\Sp_{2n}(R);\Q)$$
is an isomorphism.
\end{theorem}

\begin{proof} 
First, observe that because $\Sp_{2n}(R/\aaa)$ is finite, we have that:
$$[\Sp_{2n}(R):\Sp_{2n}(R,\aaa)]<\infty$$
Now we apply \cite[III.10.4]{Brown} to see that the rational homology of $\Sp_{2n}(R)$ is isomorphic to the $\Sp_{2n}(R)$-coinvariants of the rational homology of $\Sp_{2n}(R,\aaa)$.  In other words,
$$\HH_{\ast}(\Sp_{2n}(R);\Q)=\HH_{\ast}(\Sp_{2n}(R,\aaa);\Q)_{\Sp_{2n}(R)}$$
So, if the action of $\Sp_{2n}(R)$ on $\HH_{m}(\Sp_{2n}(R,\aaa);\Q)$ is trivial, then the desired result will follow.  Observe that \cite[Thm 9.1.11(iv)]{HOM} implies that
$$\Sp_{2n}(R)=\la \ESp_{2n}(R), \Sp_{2k}(R)\ra,$$
so we just need to show that $\ESp_{2n}(R)$ and $\Sp_{2k}(R)$ act trivially.
\paragraph{$\ESp_{2n}(R)$ acts trivially.}
Because $R/\aaa$ is finite, and $R$, as an additive group, is abelian, \cite[III.10.4]{Brown} implies that $\HH_{\ast}(\aaa;\Q)\rightarrow \HH_{\ast}(R;\Q)$ is an isomorphism, so we can apply Theorem \ref{theorem:symp} to see that $\HH_{m}(\Sp_{2(n-1)}(R,\aaa);\Q)\rightarrow\HH_{m}(\Sp_{2n}(R,\aaa);\Q)$ is a surjection.  On the other hand, our verification of assumption $3$ in our proof of Theorem \ref{theorem:symp} already showed that $\ESp_{2n}(R)$ acts trivially on the image of $\HH_{m}(\Sp_{2(n-1)}(R,\aaa);\Q)\rightarrow\HH_{m}(\Sp_{2n}(R,\aaa);\Q)$ for $n\geq k+3$, so we conclude that $\ESp_{2n}(R)$ acts trivially on $\HH_{m}(\Sp_{2n}(R,\aaa);\Q)$.
\paragraph{$\Sp_{2k}(R)$ acts trivially.}
Observe that the upper-left-included image of $\Sp_{2k}(R)$ in $\Sp_{2n}(R)$ is conjugate (by a matrix in $\ESp_{2n}(R)$) to the lower-right-included image.  The latter commutes with $\Sp_{2(n-k)}(R,\aaa)$, and hence acts trivially on
$\HH_{m}(\Sp_{2(n-k)}(R,\aaa);\Q)$.  But the map
$$\HH_{m}(\Sp_{2(n-k)}(R,\aaa);\Q)\rightarrow \HH_{m}(\Sp_{2n}(R,\aaa);\Q)$$
is surjective, hence $\Sp_{2k}(R)$ acts trivially on $\HH_{m}(\Sp_{2n}(R,\aaa);\Q)$ as desired.  (The idea behind this proof is due to Charney.  See \cite{Charney}[\S 5.1]).
\end{proof}

\noindent
David Bruce Cohen

Department of Mathematics

Rice University, MS 136

6100 Main St.

Houston, TX 77005

E-mail: {\tt dc17@rice.edu}
\end{document}